\documentclass[11pt,a4paper]{article}
\usepackage[margin=1in]{geometry}
\usepackage[T1]{fontenc}
\usepackage[utf8]{inputenc}
\usepackage{tikz}
\usepackage{hyperref}
\usepackage{amsmath,amssymb,amsfonts,amsthm,amscd, bm, mathtools}
\makeatletter
\newtheorem*{rep@theorem}{\rep@title}
\newcommand{\newreptheorem}[2]{%
\newenvironment{rep#1}[1]{%
 \def\rep@title{#2 \ref{##1}}%
 \begin{rep@theorem}}%
 {\end{rep@theorem}}}
\makeatother

\makeatletter
\def\namedlabel#1#2{\begingroup
	#2%
	\def\@currentlabel{#2}%
	\phantomsection\label{#1}\endgroup
}
\makeatother
\newtheorem{thm}{Theorem}[section]
\newreptheorem{theorem}{Theorem}
\newtheorem{prop}[thm]{Proposition}

\newtheorem{lemma}[thm]{Lemma}
\newtheorem{corol}[thm]{Corollary}
\newtheorem{case}{Case}

\theoremstyle{definition}

\title{Remarks on the countable case of the Unfriendly Partition Problem}
\author{Leandro Aurichi$^*$ and Lucas Real\footnote{Institute of Mathematics and Computer Sciences, University of S\~{a}o Paulo, S\~{a}o Paulo, Brazil}}
\date{November 2024}

\begin{document}
\maketitle
\begin{abstract}
    The Unfriendly Partition Problem asks whether it is possible to split the vertex set of a graph $G$ into two parts so that every vertex has at least as many neighbors in the other part than on its own. Despite the uncountable counterexamples provided by Milner and Shelah in 1990, this question still has no solution for graphs on countably many vertices. Under this hypothesis, our main result claims that such a bipartition exists if the rays of $G$ do not pass through infinitely many vertices of finite degree and infinitely many vertices of infinite degree simultaneously. In particular, for the class of countable graphs, we generalize previous results due to Aharoni, Milner and Prikry and due to Bruhn, Diestel, Georgakopolous and Sprüssel.      
\end{abstract}

\section{Introduction}

\paragraph{}
In this paper, a \textbf{coloring} for a graph $G$ means a bipartition of its vertex set given by a map $c:V(G) \to 2$, where $2$ is understood as the ordinal $2:=\{0,1\}$. We say that $c$ is \textbf{unfriendly} in a vertex $v\in V(G)$ if $|\{u\in N(v):c(u) \neq c(v)\}|\geq |\{u\in N(v) : c(u) = c(v)\}|$, in which $N(v)$ denotes the neighborhood of $v$. If $A$ is a vertex set or a subgraph of $G$, we also denote its neighborhood $\{u \in V(G)\setminus A : u\in N(v) \text{ for some vertex }v\in A\}$ by $N_G(A)$ (or simply by $N(A)$ if $G$ can be identified from the context). When $A\subseteq V(G)$ is a vertex set, the subgraph that it induces in $G$ shall be denoted by $G[A]$.

In the unpublished paper \cite{cowanemerson}, Cowan and Emerson conjectured that every graph admits an \textbf{unfriendly partition}, namely, a coloring which is unfriendly in every vertex. This was disproved in general by Milner and Shelah in \cite{shelah}, where these author presented a family of uncountable graphs that cannot be colored this way. However, stating the \textbf{Unfriendly Partition Problem}, it is still not know if every countable graph admits an unfriendly partition. Under this hypothesis, our main result reads as follows:

\begin{thm}\label{main}
	Every countable graph without alternating rays admits an unfriendly partition.
\end{thm}

In the above setting, we recall that a \textbf{ray} in a graph $G$ is an one-way infinite path of the form $r = v_0v_1v_2\dots$. In this case, we say that $r$ \textbf{passes} through the vertices $\{v_n\}_{n\in\mathbb{N}}\subseteq V(G)$, so that it is called an \textbf{alternating} ray if it passes through infinitely many vertices of finite degree and infinitely many vertices of infinite degree as well. In particular, Theorem \ref{main} generalizes Theorem 1 in \cite{aharoni} for countable graphs, which claims that there exist unfriendly partitions for these graphs if they contain only finitely many vertices of infinite degree. The dual statement, i.e., the existence of unfriendly partitions in countable graphs containing only finitely many vertices of finite degree, also follows from Theorem \ref{main}, although a routine greedy algorithm suffices for concluding this remark (see Exercise 22 in \cite[p. 263]{diestellivro}, for instance).

In addition, Bruhn, Georgakopoulos, Diestel and Sprüssel obtained in \cite{bruhn} that every rayless graph has an unfriendly partition. In fact, they highlighted their proof also reaches the same conclusion for countable graphs not containing rays which are attached to vertices of infinite degree in the following sense: there are infinitely many disjoint paths connecting such fixed ray to these vertices. Since this forbidden class includes the rays passing through infinitely many vertices of infinite degree, Theorem \ref{main} also generalizes this latter result. Therefore, the state of art of the Unfriendly Partition Problem for countable graphs may be summarized by Theorem \ref{main} and the main result of Berger in \cite{berger}, which sets the existence of unfriendly partitions in graphs not containing subdivisions of infinite cliques. In particular, the previous literature on the conjecture does not cover the following particular instance of Theorem \ref{main}:

\begin{corol}
	Every countable graph whose each ray passes through only finitely many vertices of finite degree admits an unfriendly partition.
\end{corol}

\section{Tools for addressing Theorem \ref{main}}

\paragraph{}

Let $T$ be a spanning tree of a graph $G$, which has a natural tree-order $\leq$ after we fix a root $r\in T$. In this case, for any $v\in V(G)$ and $A\subseteq V(G)$, we define the sets $\lceil v \rceil:=\{u\in V(G) : u \leq v\}$, $\lfloor v\rfloor :=\{u\in V(G) : u \geq v\}$, $[v]:=\lceil v \rceil \cup \lfloor v\rfloor$, $\lceil A\rceil :=\bigcup_{v\in A}\lceil v \rceil$, $\lfloor A\rfloor :=\bigcup_{v\in A}\lfloor v\rfloor$ and $[A] = \bigcup_{v\in A}[v]$. We say that $A$ is an \textbf{antichain} on $T$ if its elements are pairwise incomparable, so that $\{\lfloor v\rfloor : v \in A\}$ is a family of pairwise disjoint subsets of $T$. Naturally, the \textbf{successors} of a vertex $v\in V(G)$ are the $\leq-$minimal vertices from $\lfloor v \rfloor \setminus \{v\}$, defining the set $S(v):=N_T(v)\cap \lfloor v\rfloor$. If $u\in V(G)$ is a second vertex, the unique path in $T$ connecting it to $v$ shall be denoted by $uTv$ or $vTu$. In its turn, we recall that $T$ is called a \textbf{normal} tree if any edge of $G$ has comparable endpoints regarding $\leq$. Although not every connected graph admits normal spanning trees, this structure can be found in countable ones by recursive depth-search procedures, as first observed by Jung in \cite{jung} and outlined by Theorem 8.2.4 in \cite{diestellivro}.

Relying on the existence of normal spanning trees, the unfriendly partition claimed by Theorem \ref{main} will be constructed recursively. To that aim, it is useful to deal with partially defined colorings. More precisely, we call $c: D \to 2$ a \textbf{partial coloring} over a graph $G$ if it is defined on a subset $D\subseteq V(G)$. When $c' : D'\to 2$ is another partial coloring, we say that $c'$ and $c$ are \textbf{close} in a subset $A\subseteq D\cap D'$ if the vertex set $c\triangle c' :=\{v \in D\cap D' : c(v)\neq c'(v)\}$ is finite, contains only vertices of finite degree of $G$ and it its contained in $A$.

This definition is inherited from the work of Niel in \cite{niel}, where this author revisits the main result of Berger in \cite{berger} for countable graphs. If $c: V(G)\to 2$ is a coloring over a graph $G$ and $F\subseteq V(G)$ is a finite set of vertices of finite degree, these works also set the notations $\mathit{trans}(c) = \{uv \in E(G) : c(u)\neq c(v)\}$, $\mathit{trans}(c,F) = \{uv \in E(G): u\in F, c(u)\neq c(v)\}$ and $\mathit{dtrans}(c,F) = |\mathit{trans}(c)\setminus \mathit{trans}(c*F)|-|\mathit{trans}(c*F)\setminus \mathit{trans}(c)|$, where $c*F$ represents the coloring that differs from $c$ precisely in the vertices of $F$. In other words, $\mathit{dtrans}(c,F)$ computes the difference between the amount of edges with precisely one endpoint in $F$ whose ends are colored differently and the amount of those edges whose ends are colored the same.

If $G$ is a finite graph, any coloring $c$ that maximizes $|\mathit{trans}(c)|$ - i.e., any max-cut - is an example of an unfriendly partition. Actually, in an arbitrary graph $G$, a coloring $c$ is unfriendly in a vertex of finite degree $v\in V(G)$ if, and only if, $\mathit{dtrans}(c,\{v\})\geq 0$. Inspired by these observations, we say that $c$ is a \textbf{strongly maximal} coloring \textbf{in} some $A\subseteq V(G)$ if $\mathit{dtrans}(c,F)\geq 0$ for every finite set $F\subseteq A$ comprising vertices of finite degree. If this property is verified when $A=V(G)$, we simply call $c$ a strongly maximal coloring. In its turn, if $c$ is close in $A\subseteq V(G)$ to a strongly maximal coloring in the same vertex set $A$, then we say that $c$ is actually \textbf{almost strongly maximal} in $A$.

Especially by being preserved under minor perturbations, these maximality properties play a key role in the proof of Aharoni, Milner and Prikry in \cite{aharoni} that graphs containing only finitely many vertices of infinite degree admit unfriendly partition. In that direction, after revisiting previous works in the literature in order to grasp a convenient statement, the result below exemplify how useful strongly maximal colorings can be:

\begin{lemma}[Lemma 3 in \cite{aharoni} and Lemma 1.3 in \cite{berger}]\label{l21}
	Let $G$ be a countable graph and $c: V(G)\to 2$ be a coloring. Let $D\subseteq V(G)$ be a subset such that $c$ is strongly maximal in $V(G)\setminus D$. If $v\in D$ is a vertex of finite degree or a vertex of infinite degree in which $c$ is not unfriendly, then $c*\{v\}$ is almost strongly maximal in $V(G)\setminus D$.  
\end{lemma}
\begin{proof}[Revisited proofs of Lemma 3 in \cite{aharoni} and Lemma 1.3 in \cite{berger}]
	 By the choice of $v$, the cardinal $k:=|\{u \in N(v): c(u)\neq c(v)\}|$ is finite. For a contradiction, suppose that $c*\{v\}$ is not strongly maximal in $V(G)\setminus D$. In particular, there is some finite set $F_0\subseteq V(G)\setminus D$ comprising only vertices of finite degree such that $\mathit{dtrans}(c*\{v\},F_0)<0$. Then, the coloring $c_0:=(c*\{v\})*F_0$ now satisfies $\mathit{dtrans}(c_0,F_0)\geq 0$ by definition of $\mathit{dtrans}$ but it is still not strongly maximal in $V(G)\setminus D$. For some $n\in\mathbb{N}$, suppose that it is defined a coloring $c_n$ which is not strongly maximal in $V(G)\setminus D$ and let $F_{n+1}\subseteq V(G)\setminus D$ be a finite set of vertices of finite degree such that $\mathit{dtrans}(c_n, F_{n+1})<0$. Then, set $c_{n+1}:=c_n*F_{n+1}$ and note that $\mathit{dtrans}(c_{n+1},F_{n+1}) >0$. In addition, $c_{n+1}$ is not strongly maximal in $V(G)\setminus D$ by the assumption that $c*\{v\}$ is not almost strongly maximal within this vertex set.

	Once constructed the sequence $\{c_n\}_{n\in\mathbb{N}}$ of colorings and the sequence $\{F_n\}_{n\in\mathbb{N}}$ of finite sets of finite degree vertices, define the cardinal $\mu_n = \left|\mathit{trans}(c_n)\setminus \mathit{trans}(c*\{v\})\right|$ for each $n\in\mathbb{N}$. Since $\mathit{dtrans}(c_n, F_{n+1})<0$, it follows that $\{\mu_n\}_{n\in\mathbb{N}}$ is a strictly increasing sequence of natural numbers. In its turn, define $\hat{c}_n :=c_n*\{v\}$ for every $n\in\mathbb{N}$, so that $\hat{c}_n|_D = c|_D$ and $\hat{c}_n$ is close to $c$ in $\hat{F}_n:=\displaystyle \bigcup_{i=0}^nF_n$. For some big enough $n\in\mathbb{N}$ we must have $|\mathit{trans}(\hat{c}_n, \hat{F}_n)| > |\mathit{trans}(c,\hat{F}_n)|$, because $|\mathit{trans}(\hat{c}_n, \hat{F}_n)| - |\mathit{trans}(c,\hat{F}_n)| \geq \mu_n - k$. Therefore, it follows that $\mathit{dtrans}(c, c\triangle \hat{c}_n) < 0$ for these values of $n$, contradicting the strongly maximality assumption of $c$ in $V(G)\setminus D$.  
	  
\end{proof}

\begin{corol}\label{c22}
	Let $G$ be a countable graph and $c: V(G)\to 2$ be a coloring. Fix a subset $D\subseteq V(G)$ such that $c$ is strongly maximal in $V(G)\setminus D$. If $S\subseteq D$ is a finite set comprising vertices of finite degree or vertices of infinite degree in which $c$ is not unfriendly, then $c*S$ is almost strongly maximal in $(V(G)\setminus D)\cup S$.
\end{corol}
\begin{proof}
	Follows immediately from the above result by induction on $|S|$.
\end{proof}

For a given graph $G$, Lemma 2 in \cite{aharoni} applies compactness arguments in order to obtain a strongly maximal coloring $c: V(G)\to 2$, once we known that these bipartitions exist for finite graphs. If the set $S\subseteq V(G)$ comprising the vertices of infinite degree in which $c$ is not unfriendly is finite, then Corollary \ref{c22} shows that $c*S$ is close to a strongly maximal coloring of $G$. By the choice of $S$, this coloring is thus also an unfriendly partition. To summarize, this remark is a brief sketch of Theorem 1 in \cite{aharoni}, which claims the existence of unfriendly colorings for graphs containing only finitely many vertices of infinite degree. As another statement that is supported by compactness arguments, the following result due to Niel in \cite{niel} shall be useful in our proof for Theorem \ref{main}:

\begin{lemma}[\cite{niel}, Lemma 1.2.1.6]\label{niel}
	Let $G$ be a graph, $A\subseteq V(G)$ and $\mathcal{B} = \{B_i\}_{i\in I}$ be a family of pairwise disjoint subsets of $A$ such that, for distinct $i,j\in I$, there are no edges between vertices of $B_i$ and vertices of $B_j$. Suppose also that $N(B_i)\cap A$ is finite for every $i\in I$ and let $\tilde{c}: V(G)\to 2$ be a coloring which is strongly maximal in each such $B_i$. Then, there is a coloring $c: V(G) \to 2$ such that $c\triangle \tilde{c} \subseteq A$ comprises only vertices of finite degree and the following properties hold:
	\begin{itemize}
		\item $c$ is strongly maximal in $A$;
		\item $c|_{B_i}$ is close to $\tilde{c}|_{B_i}$ for every $i\in I$;
		\item For every finite $N\subseteq A$ there are only finitely many indexes $i\in I$ such that $\tilde{c}$ is strongly maximal in $B_i$, $N(B_i)\cap A = N$ but $c|_{B_i}\neq \tilde{c}|_{B_i}$.
	\end{itemize}
\end{lemma}
\begin{corol}\label{c25}
	Let $G$ be a graph and $c: V(G) \to 2$ be a coloring which is strongly maximal in some $A'\subseteq V(G)$. If $A\supseteq A'$ is a bigger set such that $A\setminus A'$ comprises only finitely many vertices of finite degree, then $c$ is almost strongly maximal in $A$.
\end{corol}
\begin{proof}
	This follows immediately from Lemma \ref{niel} when considering $\tilde{c} = c$ and $\mathcal{B} = \{A'\}$.
\end{proof}
\begin{corol}\label{c24}
	Let $G$ be a graph and $c: D \to 2$ be a partial coloring which is defined on a set $D\subseteq V(G)$ such that $V(G)\setminus D$ is a finite set of vertices of finite degree. If $c$ is strongly maximal in $G[D]$, then any extension to $V(G)$ is almost strongly maximal in $D$.
\end{corol}
\begin{proof}
	If $\tilde{c}$ is any extension of $c$ to $V(G)$, then the result follows immediately from Lemma \ref{niel} when considering $A = D$ and $\mathcal{B} = \{D\setminus N(V(G)\setminus D)\}$, observing that $(V(G)\setminus D)\cup N(V(G)\setminus D)$ is a finite set since $V(G)\setminus D$ contains only finitely many vertices of finite degree. In fact, $\mathit{dtrans}(\tilde{c},F) = \mathit{dtrans}(c,F)\geq 0$ if $F\subseteq V(G)$ is a finite set of vertices of finite degree which does not contain a neighbor of a vertex in $V(G)\setminus D$.
\end{proof}

\section{Proof of Theorem \ref{main}}

\paragraph{}
Bruhn, Diestel, Georgakopoulos and Sprüssel in \cite{bruhn} not only verify that rayless graphs admit unfriendly partition, but actually conclude that some partial colorings can be extended to remaining vertices in an unfriendly way. In a similar approach, we shall prove the following technical strengthening of Theorem \ref{main}, which is recovered when considering $K = \emptyset$ and observing that strongly maximal colorings are unfriendly in vertices of finite degree:

\begin{thm}\label{EnunciadoTecnico}
	Let $G$ and $K$ be two countable graphs, in which $G$ contains no alternating rays and $K$ is finite. Let $\hat{G}$ be any graph obtained from $G$ and $K$ after connecting them through arbitrary edges between $V(G)$ and $V(K)$. If $c_K: V(K)\to 2$ is any fixed partial coloring, there is an extension $\hat{c}: V(\hat{G})\to 2$ which is strongly maximal in $V(G)$ and unfriendly in its vertices of infinite degree.
\end{thm}

Throughout this section, we will then fix $G$, $K$ and $\hat{G}$ as in the above statement. In particular, since $K$ is finite, a vertex $v\in V(G)$ has infinite (resp. finite) degree in $G$ if, and only if, it has in $\hat{G}$. Therefore, we may refer to these vertices simply as vertices of infinite (resp. finite) degree. Since $G$ is countable, we will also fix a normal spanning tree $T$, whose root shall be denote by $r$. Introducing an hierarchy over the elements of $V(G)$, we will say that a vertex $v\in V(G)$ has $T$-\textbf{rank} $\mathrm{rank}_T(v)=0$ if $\lfloor v \rfloor$ comprises only vertices of finite degree or only vertices of infinite degree. Recursively, we then say that a vertex of infinite (resp. finite) degree $v\in V(G)$ has a $T-$\textbf{rank} $\mathrm{rank}_T(v) = \alpha$ if it has no smaller rank already defined but, instead, the vertices of finite (resp. infinite) degree in $\lfloor v \rfloor \setminus \{v\}$ do have smaller rank. Thus, the main hypothesis over $G$ ensures the following property:

\begin{lemma}\label{rank}
	$\mathrm{rank}_T$ is well-defined for every vertex of $G$.
\end{lemma}
\begin{proof}
	Suppose that there is a vertex $v_0 \in V(G)$ which has no assigned $T$-rank. Assuming that $v_0 \in V(G)$ has finite degree, there must exist a vertex $v_1 > v_0$ of infinite degree which also has no assigned $T$-rank, by the recursive definition of $\mathrm{rank}_T$. Similarly, there must exist now an unranked vertex $v_2 > v_1$ of finite degree. Proceeding this way, we may recursively find a chain $v_0 < v_1 < v_2 < v_3 < \dots$ of unranked vertices in $T$ such that $\{v_{2n}\}_{n\in\mathbb{N}}$ is a set of finite degree vertices and $\{v_{2n+1}\}_{n\in\mathbb{N}}$ comprises only vertices of infinite degree. Therefore, any ray $r$ on $T$ containing $\{v_n\}_{n\in\mathbb{N}}$ is an alternating ray in $G$, which contradicts the choice of this graph as in Theorem \ref{EnunciadoTecnico}.
\end{proof}

Then, we define the $T$-\textbf{rank} $\mathrm{rank}_T(G)$ of $G$ to be the $T-$rank of the root $r$. More generally, we set the \textbf{rank} of $G$ as $\mathrm{rank}(G) := \min \{\mathrm{rank}_T(G) : T\text{ is a normal spanning tree of }G\}$, so that the proof of Theorem \ref{EnunciadoTecnico} will actually be established by induction on this ordinal. Despite that, we will indeed assume that the previously fixed tree $T$ minimizes $\mathrm{rank}_T(G)$, witnessing the value of $\mathrm{rank}(G)$. If $\mathrm{rank}_T(G) = \mathrm{rank}(G) = 0$, then $G$ is either a locally finite or a $\aleph_0-$regular graph. In other words, its vertices have all finite degree or all infinite degree. In the former case, Theorem \ref{EnunciadoTecnico} is reduced to Lemma 2 in \cite{aharoni}. In the latter one, Theorem \ref{EnunciadoTecnico} claims the existence of an unfriendly partition for $G$, which can be constructed iteratively via a greedy algorithm (see Exercise 22 in \cite[[p. 263]{diestellivro}).

Therefore, we may assume that $\alpha := \mathrm{rank}_T(G) >0$ and that Theorem \ref{EnunciadoTecnico} holds if $G$ is replaced by a subgraph of smaller rank. Then, the lemma below highlights a first instance of the inductive argument we aim to develop:

\begin{lemma}\label{RaizGrauFinito}
	Suppose that the root $r$ of $T$ has finite degree as a vertex in $G$. If $c_K : V(K) \to 2$ is any fixed partial coloring as in Theorem \ref{EnunciadoTecnico}, there is an extension $\hat{c}: V(\hat{G})\to 2$ which is strongly maximal in $V(G)$ and unfriendly in its vertices of infinite degree. 
\end{lemma}
\begin{proof}
	Denote by $S\subseteq V(G)$ the set of $\leq-$minimal vertices of infinite degree of $G$, so that $\lceil S\rceil\setminus S $ contains only vertices of finite degree. Therefore, since $T$ is a normal tree and $\lceil v \rceil$ is finite for every $v\in V(G)$, the set $\lceil S\rceil$ induces a locally finite subgraph of $G$. In its turn, the vertex sets of the connected components of $G\setminus \lceil S\rceil$ are given by $\{\lfloor u \rfloor\}_{u\in S'}$, where $S'$ comprises the $\leq-$minimal vertices of $G\setminus \lceil S\rceil$. For convenience, we split $S'$ into the subsets $I:=\{u \in S': u> v \text{ for some }v\in S\}$ and $J:=S'\setminus I$. Once $S$ is defined as the set containing all $\leq-$minimal vertices of infinite degree in $G$, all the vertices from $\bigcup_{u \in J}\lfloor u \rfloor $ have also finite degree. Therefore, the subgraph $H$ of $G$ induced by $\lceil S\rceil \cup \bigcup_{u\in J}\lfloor u \rfloor $ is locally finite and the vertex sets of the connected components of $G\setminus H$ are given by $\{\lfloor u\rfloor\}_{u\in I}$. We then fix any arbitrary partial coloring $c_{H}: V(K)\cup V(H) \to 2$ such that $c_H|_{V(K)} = c_K$.

	Given $u\in I$, we now observe that $T_u := T[\lfloor u \rfloor]$ is a normal spanning tree for $B_u :=G[\lfloor u \rfloor]$, since $T$ is a normal spanning tree for $G$. On the other hand, by definition of $I$, there is a vertex $v\in S'$ such that $v > u$. Hence, $\mathrm{rank}(B_u)\leq \mathrm{rank}_{T_u}(B_u) = \mathrm{rank}_T(u)\leq \mathrm{rank}_T(v) < \mathrm{rank}_T(r) = \mathrm{rank}(G)$, where the strict inequality in this expression follows from the fact that $r$ has finite degree by assumption while $N(v)$ is infinite. Then, Theorem \ref{EnunciadoTecnico} claims (by induction) that there is a coloring $\hat{c}_u : V(B_u) \cup N_{\hat{G}}(B_u) \to 2$ which agrees with $c_{H}$ in $N_{\hat{G}}(B_u)\subseteq \lceil v\rceil \cup V(K)$, it is strongly maximal in $B_u$ and it is unfriendly in its vertices of infinite degree. In particular, the coloring $\hat{c}:= c_H\cup \bigcup_{u\in I}\hat{c}_u$ is well-defined over $V(\hat{G})$.

	Now, let $X\subseteq S$ denote the set of vertices for which $\hat{c}$ is not unfriendly and fix $v\in X$. According to Lemma \ref{l21}, $\hat{c}*\{v\}$ is almost strongly maximal in $\lfloor v \rfloor \setminus \{v\} = \displaystyle \bigcup_{\substack{u\in I \\ u > v}}\lfloor u \rfloor $. Hence, let $\hat{\hat{c}} : V(\hat{G})\to 2$ be a coloring such that:
	
	\begin{itemize}
		\item $\hat{\hat{c}}$ agrees with $\hat{c}*X$ in $V(K)\cup V(H)\cup\bigcup_{v \in S\setminus X}\lfloor v\rfloor$. In particular, $\hat{\hat{c}}$ is strongly maximal in $\lfloor v \rfloor \setminus \{v\}$ for every $v\in S\setminus X$;
		\item $\hat{\hat{c}}$ is strongly maximal in $\lfloor v \rfloor \setminus \{v\}$ for every $v\in X$ and, within this subset, it differs from $\hat{c}$ in only finitely many vertices of finite degree;
		\item $\hat{\hat{c}}$ agrees with $c*X$ in the vertices of infinite degree of $G$. In particular, the definition of $X$ and the above two items ensure that $\hat{\hat{c}}$ is unfriendly in the vertices of $S$, since every $v\in V(G)$ has only finitely many neighbors out of $\lfloor v \rfloor \setminus \{v\}$. Similarly, $\hat{\hat{c}}$ is unfriendly in the vertices of infinite degree that do not belong to $S$ because so does $\hat{c}$ by its construction.
	\end{itemize}     

Therefore, after applying Lemma \ref{niel} to the coloring $\hat{\hat{c}}$, the set $A=V(\hat{G})$ and the family $\mathcal{B}:=\{\lfloor v\rfloor \setminus \{v\}: v \in S\}$, we obtain a strongly maximal coloring $c: V(\hat{G})\to 2$ as required by Theorem \ref{EnunciadoTecnico}. In fact, $c$ is unfriendly in the vertices of $S$ since $c|_{\lfloor v \rfloor}\triangle \hat{\hat{c}}|_{\lfloor v \rfloor}$ is a finite set of vertices of finite degree. By the same reason, $c$ is unfriendly in the vertices of infinite degree that do not belong to $S$ once so does $\hat{\hat{c}}$.  
\end{proof}

Hence, due to Lemma \ref{RaizGrauFinito}, it remains to consider the case in which the root $r$ of $T$ has infinite degree. Similar to the beginning of the above proof, let then $S\subseteq V(G)$ denote the set of $\leq-$minimal vertices of finite degree in $G$. If $u$ is a successor of some $v \in S$, then $T_u:=T[\lfloor u \rfloor]$ is again a normal spanning tree for $B_u:=G[\lfloor u \rfloor]$. Now,  $\mathrm{rank}(B_u)\leq \mathrm{rank}_{T_u}(B_u) = \mathrm{rank}_T(u)\leq \mathrm{rank}_T(v) < \mathrm{rank}_T(r) = \mathrm{rank}(G)$, where the strict inequality in this expression follows from the fact that $r$ has infinite degree by assumption while $N(v)$ is finite. In particular, the following instance of Theorem \ref{EnunciadoTecnico} holds:

\begin{lemma}\label{l34}
	For a given $X \subseteq S$, let $D\subseteq V(\hat{G})$ be a set containing $V(K)$ and disjoint from $\lfloor X\rfloor $. Then, in the subgraph induced by $D\cup \lfloor X\rfloor$, any coloring $c' : D \to 2$ can be extended to a coloring $c: D \cup \lfloor X\rfloor \to 2$ which is strongly maximal in $\lfloor X\rfloor$;
		 
\end{lemma}

\begin{proof}
	Let $c': D \to 2$ be a coloring as in the statement and fix any bipartition $c_X : X \to 2$ over $X$. Given $u\in\bigcup_{v\in X} S(v)$ and recalling that $\mathrm{rank}(B_u)< \mathrm{rank}(G)$, Theorem \ref{EnunciadoTecnico} claims (by induction) that there is an extension $c_u: N(B_u)\cap (D\cup \{v\})\cup V(B_u)\to 2$ to the coloring $c'\cup c_X|_{N(B_u)\cap (D\cup \{v\})}$ which is strongly maximal in $\lfloor u\rfloor$ when considering the graph induced by $N(B_u)\cap (D\cup \{v\})\cup V(B_u)$. Therefore, regarding the subgraph induced by its domain, the coloring $c_v':=c'\cup c_X\cup \bigcup_{u\in S_v}c_u$ is strongly maximal in $\lfloor S_v\rfloor$ for every $v\in X$, since $\{\lfloor u \rfloor : u \in S_v \}$ is the family of connected subsets of $\lfloor S_v \rfloor$. In its turn, Corollary \ref{c25} and the finite degree of $v$ ensure that $c_v'$ is close in $\lfloor v \rfloor$ to a coloring $c_v$ which is strongly maximal in this subset. Since $\{\lfloor v \rfloor : v \in X\}$ is a family of pairwise disjoint subsets of $G$ and there are no edges connecting two distinct of them (by normality of $T$), the coloring $c:= c'\cup \bigcup_{v\in X}c_v$ is thus well-defined and strongly maximal in $\lfloor X\rfloor$.   
\end{proof}

In order to continue the proof, we define the set $H:=V(G)\setminus \lfloor S\rfloor$, noticing that it contains only vertices of infinite degree in $G$ by definition of $S$. Then, we shall say that a partial coloring $c: D \to 2$ is \textbf{stable} if the following properties are verified:

\begin{itemize}
	\item \namedlabel{ps1}{\texttt{Property S1)}}: The domain of $c$ can be written as a disjoint union $D= H_c \cup \lfloor X_c\rfloor$, where $H_c\subseteq H$ and $X_c\subseteq H\cup S$ is an antichain on the tree order of $T$. In particular, if a vertex $v\in D$ has finite degree, then $v\in \lfloor X_c\rfloor $;
	\item \namedlabel{ps2}{\texttt{Property S2)}}: As a coloring defined over $\hat{G}[D]$, $c$ is strongly maximal in $\lfloor X_c\rfloor$. By definition of $H$ and the normality of $T$, this is equivalent to the following statements: $(i)$ $c$ is strongly maximal in $\lfloor v\rfloor$ for every $v \in S\cap D$ and $(ii)$ $c$ is strongly maximal in $\lfloor w\rfloor$ for every $w\in X'$, where $X'\subseteq D$ is any antichain on $T$ such that $S\cap D \subseteq \lfloor X'\rfloor$;
	\item \namedlabel{ps4}{\texttt{Property S3)}}: Every $w\in V(\hat{G})\setminus D$ has only finitely many neighbors above a given $v\in X_c$, i.e., the set $N(w)\cap \lfloor v\rfloor$ is finite;
\end{itemize}

The main feature of these colorings is that they can be extended while controlling their patterns in some up-closed subsets of $V(\hat{G})$. In that regard, the following technical result will support many constructive arguments from now on in this section:

\begin{lemma}\label{fixing}
	Let $c: D \to 2$ be a stable partial coloring of $\hat{G}$ and fix an extension $c': D'\to 2$ that is strongly maximal in $D'\setminus D\subseteq H$. Then, there is a coloring $\tilde{c}: D'\to 2$ which is strongly maximal in $\lfloor X_c\rfloor$, verifies $\tilde{c}\triangle c'\subseteq \lfloor X_{c}\rfloor$ and such that $(\tilde{c}\triangle c')\cap \lfloor v\rfloor$ is a finite set of vertices of finite degree for every $v\in X_{c}$. Moreover, we can assume that $\tilde{c}|_{\lfloor v \rfloor} = c|_{\lfloor v\rfloor}$ if $N_{\tilde{G}}(\lfloor v\rfloor)\cap (D'\setminus D) = \emptyset$.
\end{lemma}
\begin{proof}
	Given $v\in X_{c}$, \ref{ps4} in the definition of stable colorings asserts that every $w\in D'\setminus D$ has finite degree in the graph $G_v:=\hat{G}[D'\cap [v]]$. In particular, applied to this graph and the sets $A' = \lfloor v\rfloor$ and $A = D'\cap [v]$, Corollary \ref{c25} claims the existence of a strongly maximal coloring $c_v : V(G_v)\to 2$ which is close to $c'|_{V(G_v)}$ in $\lfloor v\rfloor$. Indeed, we can choose $c_v=c'|_{V(G_v)}$ if no vertex from $D'\setminus D$ has a neighbor in $\lfloor v\rfloor$. On the other hand, note that $V(G_v)$ contains $N_{\hat{G}}(\lfloor v\rfloor)\cap D'$ by the normality of $T$. Hence, $\tilde{c}:=c'|_{D'\setminus [X_{\alpha}]}\cup \bigcup_{v\in X_{\alpha}}c_v$ is a well-defined strongly maximal coloring as claimed by the statement.  
\end{proof}

 For some big enough ordinal $\Omega$, we will now construct a sequence of partial stable colorings $\{c_{\alpha}\}_{\alpha \leq \Omega}$ such that $c_{\Omega}: V(\hat{G})\to 2$ may be the coloring claimed by Theorem \ref{EnunciadoTecnico}. For every $\alpha < \Omega$, the domain of $c_{\alpha}$ will be denoted by $D_{\alpha}$ and shall be written as a disjoint union of the form $D_{\alpha}:=H_{\alpha}\cup \lfloor X_{\alpha}\rfloor$, where, as required by \ref{ps1} in the definition of stable colorings, $H_{\alpha}$ is a subset of $ H$ and $X_{\alpha}\subseteq H\cup S$ is an antichain on the tree-order of $T$. If $v \in D_{\beta}$ has infinite degree in $\hat{G}$ for some $\beta < \alpha$, we will also guarantee that $c_{\beta}(v) = c_{\alpha}(v)$.

First, we consider $c_0: D_0 \to 2$ to be an extension of $c_K$ which is also an unfriendly partition for the graph induced by the maximal $\aleph_0-$regular subset $H_0\subseteq H$. In other words, $H_0$ is the $\subseteq-$maximal subset of $H$ such that $N(w)\cap H_0$ is infinite for every $w\in H_0$, whose existence naturally follows from Zorn's Lemma. Then, as mentioned after the proof of Lemma \ref{rank}, $c_0$ might be constructed by routine greedy algorithms. Note that this is indeed a stable coloring since \ref{ps2} and \ref{ps4} are vacuously satisfied. In its turn, assuming that $c_{\alpha}$ is defined for some ordinal $\alpha$, the coloring $c_{\alpha+1}$ and its domain will be constructed according to one of the following cases:

\begin{case}\label{case3}
	Considering the tree order of $T$, suppose that there is $w\in H\setminus D_{\alpha}$ which contains infinitely many successors in $S\setminus D_{\alpha}$. Then, $c_{\alpha+1}$ and $D_{\alpha+1}$ are chosen so that $H_{\alpha+1}:= H_{\alpha}\cup N_{\hat{G}}(\lfloor X\rfloor)$ and $X_{\alpha+1}:=X_{\alpha}\cup X$, where $X\subseteq S(w)\cap S\setminus D$ is an infinite set such that $N_{\hat{G}}(\lfloor X\rfloor)=\{w'\leq w : |N(w')\cap \lfloor X\rfloor|=\aleph_0\}$. Moreover, $\{u\in N(w')\cap \lfloor X\rfloor : c_{\alpha+1}(u)\neq c_{\alpha+1}(w')\}$ is infinite for every $w'\in H_{\alpha+1}\setminus H_{\alpha}$ (and, in particular, for $w'=w$).
	
\end{case}
\begin{proof}[Construction of $c_{\alpha+1}$]
	We first note that an infinite set $X\subseteq S(w)\cap S\setminus D$ as in the above statement indeed exists since $\lceil w\rceil$ is finite and, by the normality of $T$, it contains $N_{\hat{G}}(\lfloor X\rfloor)$. In this case, $X_{\alpha+1} = X_{\alpha}\cup X$ is indeed an antichain on $T$ because so it is $X_{\alpha}$ by induction and $X\cap D_{\alpha} = \emptyset$. Moreover, observing that $X$ contains only successors of $w$, we highlight that $w\in N_{\hat{G}}(\lfloor X\rfloor)\subseteq D_{\alpha}$.

		On the other hand, Lemma \ref{l34} provides a coloring $c_{\alpha}': D_{\alpha+1}\to 2$ which extends $c_{\alpha}$ and is strongly maximal in $\lfloor X\rfloor$. After changing the colors of finitely many vertices if necessary, we can assume that $|\{u\in N(w')\cap \lfloor X\rfloor: c_{\alpha}'(u)\neq c_{\alpha}'(w')\}|=\aleph_0$, as claimed by Corollary \ref{c22} when applied to the graph $\hat{G}[D_{\alpha+1}\cap [w]]$ and the set $D = N_{\hat{G}}(\lfloor X\rfloor)$.

		Once $c_{\alpha}$ is a stable coloring, Lemma \ref{fixing} now claims the existence of a coloring $c_{\alpha+1}: D_{\alpha+1}\to 2$ which is strongly maximal in $\lfloor X_{\alpha}\rfloor$, verifies $c_{\alpha+1}\triangle c_{\alpha}'\subseteq \lfloor X_{\alpha}\rfloor$ and such that $(c_{\alpha+1}\triangle c_{\alpha}')\cap \lfloor v \rfloor$ is a finite set of vertices of finite degree for every $v\in X_{\alpha}$. Moreover, this intersection is empty if $N_{\hat{G}}(\lfloor v\rfloor)\cap (H_{\alpha+1}\setminus H_{\alpha}) = \emptyset$. Since $c_{\alpha+1}$ agrees with $c_{\alpha}'$ in $\lfloor X\rfloor$, it follows that $c_{\alpha+1}$ is actually strongly maximal in $\lfloor X_{\alpha+1}\rfloor$, proving that this coloring verifies \ref{ps2}. By the choice of the vertices in $H_{\alpha+1}\setminus H_{\alpha}$, \ref{ps4} is also satisfied by $c_{\alpha+1}$, so that this is indeed a stable coloring. Finally, we remark that $c_{\alpha+1}$ and $c_{\alpha}$ agree on the vertices of infinite degree belonging to $D_{\alpha}$. 
\end{proof}

\begin{case}\label{case1}
	Suppose that there is $w\in H\setminus D_{\alpha}$ for which there is $v\in S\setminus D_{\alpha}$ such that $N(w)\cap\lfloor v\rfloor$ is infinite. Then, $c_{\alpha+1}$ and $D_{\alpha+1}$ are chosen so that $H_{\alpha+1}:= H_{\alpha}\cup \{w'\in H\setminus D_{\alpha}: |N(w)\cap \lfloor v\rfloor|=\aleph_0\}$ and $X_{\alpha+1}:=X_{\alpha}\cup \{v\}$. Moreover, $\{u\in N(w')\cap \lfloor v\rfloor : c_{\alpha+1}(u)\neq c_{\alpha+1}(w')\}$ is infinite for every $w'\in H_{\alpha+1}\setminus H_{\alpha}$ (and, in particular, for $w'=w$).
\end{case}
\begin{proof}[Construction of $c_{\alpha+1}$]
	We first remark that $X_{\alpha+1}$ as defined above is indeed an antichain on $T$ because so it is $X_{\alpha}$ and since $v\in S\setminus D_{\alpha}$. Now, by applying the first item of Lemma \ref{l34} when considering $X = \lfloor v \rfloor$, there is an extension $c_{\alpha}': D_{\alpha+1}\to 2$ of $c_{\alpha}$ which is strongly maximal in $\lfloor v \rfloor$. After possibly changing the colors of finitely many vertices of finite degree and finitely many vertices of $H_{\alpha+1}\setminus H_{\alpha}$, we can assume that $|\{u\in N(w')\cap \lfloor v\rfloor : c_{\alpha}'(w')\neq c_{\alpha}'(u)\}| = \aleph_0$ for every $w'\in H_{\alpha+1}\setminus H_{\alpha}$, as claimed by Corollary \ref{c22} when applied to the subgraph $\hat{G}[[v]\cap D_{\alpha+1}]$ and considering $D = H_{\alpha+1}\setminus H_{\alpha}$.

	Once $c_{\alpha}$ is a stable coloring, Lemma \ref{fixing} now applies to $c_{\alpha}'$ and claims the existence of a coloring $c_{\alpha+1}$ which is strongly maximal in $\lfloor X_{\alpha}\rfloor$, verifies $c_{\alpha+1}\triangle c_{\alpha}'\subseteq \lfloor X_{\alpha}\rfloor$ and such that $(c_{\alpha+1}\triangle c_{\alpha}')\cap \lfloor v' \rfloor$ is a finite set of vertices of finite degree for every $v'\in X_{\alpha}$. In addition, this latter intersection is empty if $N_{\hat{G}}(\lfloor v'\rfloor)\cap (H_{\alpha+1}\setminus H_{\alpha}) = \emptyset$. Since $c_{\alpha+1}$ agrees with $c_{\alpha}'$ in $\lfloor v\rfloor$, the former coloring is actually strongly maximal in $\lfloor X_{\alpha+1}\rfloor$, so that it verifies \ref{ps2}. By the choice of the vertices in $H_{\alpha+1}\setminus H_{\alpha}$, \ref{ps4} also holds for $c_{\alpha+1}$, which is thus also a stable coloring. Finally, we note that $c_{\alpha}$ and $c_{\alpha+1}$ agree on the vertices of infinite degree of $D_{\alpha}$.
	
\end{proof}

\begin{case}\label{case5}
	Suppose that there is a vertex $w\in H\setminus D_{\alpha}$ such that $(\lfloor w\rfloor \setminus \{w\})\subseteq D_{\alpha}$. Then, $c_{\alpha+1}$ and $D_{\alpha+1}$ are chosen so that $X_{\alpha+1}:=(X_{\alpha}\setminus \lfloor w\rfloor)\cup \{w\}$ and $H_{\alpha+1}:=H_{\alpha}\cup I_w$, where $I_w:=\{w'\in H\setminus D_{\alpha}: |N(w')\cap \lfloor w\rfloor|=\aleph_0\}$. Moreover, $\{u\in N(w')\cap \lfloor w\rfloor : c_{\alpha+1}(u)\neq c_{\alpha+1}(w')\}$ is infinite for every $w'\in I_w$ (and, in particular, for $w'=w$).
\end{case}
\begin{proof}[Construction of $c_{\alpha+1}$]
	When restricted to $[w]\cap D_{\alpha}$, we now observe that $c_{\alpha}$ gives rise to a stable coloring $c_w:=c_{\alpha}|_{D_{\alpha}\cap [w]}$, where \ref{ps1} is verified by the sets $H_w:=H_{\alpha}\cap [w]$ and $X_w:=X_{\alpha}\cap \lfloor w\rfloor$. Then, after fixing any extension $c_w'$ of $c_w$ to $(D_{\alpha}\cap [w])\cup I_w$, Lemma \ref{fixing} provides a coloring $\tilde{c}_w: (D_{\alpha}\cap[w])\cup I_w \to 2$ which is strongly maximal in $\lfloor X_w\rfloor$, verifies $\tilde{c}_w\triangle c_w'\subseteq \lfloor X_w\rfloor$ and such that $(\tilde{c}_w\triangle c_w')\cap \lfloor w\rfloor$ is a finite set of vertices of finite degree for every $v\in X_w$. Moreover, we can assume that this latter intersection is empty if $N_{\hat{G}}(\lfloor v\rfloor)\cap I_w = \emptyset$.

	In its turn, by applying Lemma \ref{l34} to the set $X = S\cap \lfloor w\rfloor \setminus D_{\alpha}$, we can extend $\tilde{c}_w$ to the vertices of $\lfloor X\rfloor$ in order to assume that this coloring is actually defined over $D_{\alpha+1}\cap [w]$ and it is strongly maximal in $\lfloor X_{w}\cup X\rfloor$. Besides that, unless by changing the values on $\tilde{c}_u$ in vertices of $I_u$ and finitely many vertices of finite degree in $\lfloor u\rfloor$, Corollary \ref{c22} allows us to assume that every vertex of $I_w$ has infinitely many neighbors of opposite color in $\lfloor w\rfloor$.

	 Thus, $c_{\alpha}':=c_{\alpha}|_{D_{\alpha}\setminus [w]}\cup \tilde{c}_w$ is an extension of $c_{\alpha}|_{D_{\alpha}\setminus (I_w\cup \lfloor w\rfloor)}$, which we can now regard as a stable coloring whose domain can be decomposed as $(H_{\alpha}\setminus I_w)\cup \lfloor X_{\alpha+1}\setminus \{w\} \rfloor$. Applying Lemma \ref{fixing} to this setting, we can choose $c_{\alpha+1}$ as a strongly maximal coloring in $\lfloor X_{\alpha+1}\setminus \{w\}\rfloor$ such that $c_{\alpha+1}\triangle c_{\alpha}'\subseteq \lfloor X_{\alpha+1}\rfloor$ and $(c_{\alpha+1}\triangle c_{\alpha}')\cap \lfloor v\rfloor$ is finite for every $v\in X_{\alpha+1}\setminus \{w\}$. Once $c_{\alpha}'$ is strongly maximal in $\lfloor w\rfloor$ by construction, it follows that $c_{\alpha+1}$ is actually strongly maximal in $\lfloor X_{\alpha+1}\rfloor$, as required by \ref{ps2}. In addition, $c_{\alpha+1}$ verifies \ref{ps4} by definition of $I_w$, so that this coloring is indeed a stable one. Finally, note that $c_{\alpha}$ and $c_{\alpha+1}$ agree on the vertices of infinite degree belonging to $D_{\alpha}$.
\end{proof}

\begin{case}\label{case2}
	Suppose that there is a vertex $w\in H\setminus D_{\alpha}$ that contains infinitely many neighbors in $H\cap D_{\alpha}$. Then, $c_{\alpha+1}$ and $D_{\alpha+1}$ are chosen so that $X_{\alpha+1}:=X_{\alpha}$ and $D_{\alpha+1}:=D_{\alpha}\cup \{w\}$. Moreover, the set $\{u\in N(w)\cap H\cap D_{\alpha}: c_{\alpha+1}(w)\neq c_{\alpha+1}(u)\}$ is infinite.
\end{case}
\begin{proof}[Construction of $c_{\alpha+1}$]
	Since $N(w)\cap H\cap D_{\alpha}$ is infinite, there is $i\in 2$ for which $|N(w)\cap H\cap c_{\alpha}^{-1}(i)| = \aleph_0$, so that we define $c_{\alpha}'$ to be the extension of $c_{\alpha}$ to $D_{\alpha+1}:=D_{\alpha}\cup \{w\}$ given by $c_{\alpha}'(w) = 1-i$. Since $c_{\alpha}$ is a stable coloring, Lemma \ref{fixing} claims the existence of a strongly maximal coloring $c_{\alpha'}: D_{\alpha+1}\to 2$ such that $c_{\alpha+1}\triangle c_{\alpha}'\subseteq \lfloor X_{\alpha}\rfloor$ and $(c_{\alpha}\triangle c_{\alpha}')\cap \lfloor v\rfloor$ is a finite set of vertices of finite degree for every $v\in X_{\alpha}$. In addition, this latter intersection is empty if $w$ does not contains a neighbor in $\lfloor v \rfloor$. Then, \ref{ps2} and \ref{ps4} are verified by $c_{\alpha+1}$ immediately, because they hold for $c_{\alpha}$ by induction. Finally, we remark that $c_{\alpha}$ and $c_{\alpha+1}$ agree on the vertices of infinite degree belonging to $D_{\alpha}$.
\end{proof}

\begin{case}\label{case4}
	Considering the tree order of $T$, suppose that there is $w\in H\setminus D_{\alpha}$ for which $\lfloor u\rfloor \subseteq D_{\alpha}$ and $|N(w)\cap \lfloor u\rfloor| = \aleph_0$ for some successor $u\in S(w)$. Then, $c_{\alpha+1}$ and $D_{\alpha+1}$ are chosen so that $X_{\alpha+1}:=(X_{\alpha}\setminus \lfloor u\rfloor)\cup\{u\}$ and $H_{\alpha+1}:=H_{\alpha}\cup I_w$, where $I_w = \{w'\in H\setminus D_{\alpha}: |N(w')\cap \lfloor u\rfloor|=\aleph_0\}$. Moreover, $\{u'\in N(w')\cap \lfloor u\rfloor : c_{\alpha+1}(u')\neq c_{\alpha+1}(w')\}$ is infinite for every $w'\in I_w$ (and, in particular, for $w'=w$).
\end{case}
\begin{proof}[Construction of $c_{\alpha+1}$]
	The definition of $c_{\alpha+1}$ is similar to the one provided by Case \ref{case5}. When restricted to $[u]\cap D_{\alpha}$, we first observe that $c_{\alpha}$ gives rise to a stable coloring $c_u:=c_{\alpha}|_{D_{\alpha}\cap [u]}$, where \ref{ps1} is verified by the sets $H_u:=H_{\alpha}\cap [u]$ and $X_u:=X_{\alpha}\cap \lfloor u\rfloor$. Then, after fixing any extension $c_u'$ of $c_u$ to $(D\cap [u])\cup I_w$, Lemma \ref{fixing} provides a coloring $\tilde{c}_u: D_{\alpha}\cap[u]\cup I_w \to 2$ which is strongly maximal in $\lfloor X_u\rfloor$, verifies $\tilde{c}_u\triangle c_u'\subseteq \lfloor X_u\rfloor$ and such that $(\tilde{c}_u\triangle c_u')\cap \lfloor v\rfloor$ is a finite set of vertices of finite degree for every $v\in X_u$. Moreover, we can assume that this latter intersection is empty if $N_{\hat{G}}(\lfloor v\rfloor)\cap I_w = \emptyset$. We even remark that $\tilde{c}_u$ is strongly maximal in $\lfloor u\rfloor$, since the vertices of finite degree in this set actually lie on $\lfloor X_u\rfloor$. In its turn, unless by changing the values on $\tilde{c}_u$ in vertices of $I_w$ and finitely many vertices of finite degree in $\lfloor u\rfloor$, Corollary \ref{c22} allows us to assume that every vertex of $I_w$ has infinitely many neighbors of opposite color in $\lfloor u\rfloor$.

	Thus, $c_{\alpha}':=c_{\alpha}|_{D_{\alpha}\setminus [u]}\cup \tilde{c}_u$ is an extension of $c_{\alpha}|_{D_{\alpha}\setminus (I_w\cup \lfloor u\rfloor)}$, which we can now regard as a stable coloring whose domain can be decomposed as $(H_{\alpha}\setminus I_w)\cup \lfloor X_{\alpha+1}\setminus \{u\} \rfloor$. Applying Lemma \ref{fixing} to this setting, we can choose $c_{\alpha+1}$ as a strongly maximal coloring in $\lfloor X_{\alpha+1}\setminus \{u\}\rfloor$ such that $c_{\alpha+1}\triangle c_{\alpha}'\subseteq \lfloor X_{\alpha+1}\rfloor$ and $(c_{\alpha+1}\triangle c_{\alpha}')\cap \lfloor v\rfloor$ is finite for every $v\in X_{\alpha+1}\setminus \{u\}$. Once $c_{\alpha}'$ is strongly maximal in $\lfloor u\rfloor$ by construction, it follows that $c_{\alpha+1}$ is actually strongly maximal in $\lfloor X_{\alpha+1}\rfloor$, as required by \ref{ps2}. In addition, $c_{\alpha+1}$ verifies \ref{ps4} by definition of $I_u$, so that this coloring is indeed a stable one. Finally, note that $c_{\alpha}$ and $c_{\alpha+1}$ agree on the vertices of infinite degree belonging to $D_{\alpha}$.
\end{proof}

In its turn, suppose now that $\alpha$ is a limit ordinal and that $c_{\beta}$ is defined for every $\beta < \alpha$. The domain of $c_{\alpha}$ will be given by $D_{\alpha}:=\bigcup_{\beta < \alpha}D_{\beta}$, which we can decompose as the union between $H_{\alpha}:=D_{\alpha}\setminus \lfloor X_{\alpha}\rfloor \subseteq H$ and $\lfloor X_{\alpha}\rfloor$, where $X_{\alpha}$ comprises the $\leq-$minimal vertices of $ \bigcup_{\beta < \alpha}X_{\beta}$. For a vertex $w\in H_{\alpha}$, we fix the ordinal $\beta_w=\min\{\beta < \alpha : w\in D_{\beta}\}$ and define $c_{\alpha}(w) :=c_{\beta_w}(w)$, observing by induction that $c_{\beta_w}(w) = c_{\beta}(w)$ for every $\beta_w \leq \beta <\alpha$. Similarly, we fix the ordinal $\beta_v=\min\{\beta < \alpha : v\in X_{\beta}, N_{\hat{G}}(\lfloor v\rfloor)\cap D_{\alpha} = N_{\hat{G}}(\lfloor v\rfloor)\cap D_{\beta}\}$ for every $v\in X_{\alpha}$, which exists since $N_{\hat{G}}(\lfloor v\rfloor)\subseteq \lceil v\rceil$ is finite. We note that $v \in X_{\beta}$ for every $\beta_v \leq \beta < \alpha$ by the $\leq-$minimality of $v$, especially when applied to the definition of $X_{\beta+1}$ as in Cases \ref{case5} and \ref{case4}. Moreover, since $N_{\hat{G}}(\lfloor v\rfloor)\cap D_{\alpha} = N_{\hat{G}}(\lfloor v\rfloor)\cap D_{\beta_v}$, it follows that $c_{\beta}|_{\lfloor v\rfloor} = c_{\beta_v}|_{\lfloor v\rfloor}$ for every $\beta_v\leq \beta < \alpha$. Due to this reason, we define $c_{\alpha}|_{\lfloor v\rfloor}:=c_{\beta_v}|_{\lfloor v\rfloor}$. In particular, $c_{\alpha}$ arises from this definition as a strongly maximal coloring in $\lfloor X_{\alpha}\rfloor$. Analogously, $c_{\alpha}$ verifies \ref{ps4} because so it does $c_{\beta_v}$ for every $v\in X_{\alpha}$.

Since $G$ is countable, there is an ordinal $\Omega<\omega_1$ such that $D_{\Omega+1}= D_{\Omega}$, concluding that the recursive process which defines the sequence $\{c_{\alpha}\}_{\alpha \leq \Omega}$ eventually halts. Our final goal is to show that $c_{\Omega}$ holds as the coloring claimed by Theorem \ref{EnunciadoTecnico}, which first requires that this bipartition is globally defined. When assuming that $D_{\Omega}\neq V(\hat{G})$, the following observation is the core of a contradiction to be reached:

\begin{lemma}\label{l36}
	Suppose that $v\in H$ is a vertex such that $\lfloor v\rfloor\setminus D_{\Omega}\neq \emptyset$. Then, there is a vertex $w\in H\setminus D_{\Omega}$ such that $w>v$, $S(w)$ is finite and $w$ contains only finitely many neighbors in $H$.
\end{lemma} 
\begin{proof}
    We first observe that $S(w)$ is finite for every $w\in H\setminus D_{\Omega}$. Otherwise, Case \ref{case3} provides a coloring $c_{\Omega+1}$ whose domain contains $w$, contradicting the choice of $\Omega$.

	As a less immediate remark, we now claim that a given vertex $w\in H\setminus D_{\Omega}$ also contains only finitely many successors in $S\cap D_{\Omega}$. For a moment, suppose that these successors define an infinite subset $X\subseteq S\cap D_{\Omega}$, so that $X\subseteq X_{\Omega}$ by \ref{ps1} of stable colorings. In particular, the ordinal $\gamma :=\min\{\eta < \Omega : |X\cap X_{\eta}|=\aleph_0\}$ is well defined. We observe that $\gamma$ cannot be written as $\gamma = \alpha+1$ for some ordinal $\alpha$. Otherwise, Case \ref{case3} describes the only possible way to obtain $c_{\gamma}$ from $c_{\alpha}$ so that $X_{\gamma}\setminus X_{\alpha}$ is infinite, but this also implies that $w\in D_{\gamma}\subseteq D_{\Omega}$. Therefore, $\gamma$ must be a limit ordinal and we may find an increasing sequence of successor ordinals $\{\gamma_n\}_{n\in\mathbb{N}}\subseteq \gamma$ such that $X\cap X_{\gamma_n}\subsetneq X\cap X_{\gamma_{n+1}}$ for every $n\in\mathbb{N}$. Actually, we can describe this sequence recursively by setting $\gamma_{n+1}:=\min\{\eta < \gamma : X\cap X_{\gamma_n}\subsetneq X\cap X_{\eta}\}$, so that $c_{\gamma_{n+1}}$ is obtained from a previously defined coloring according to Case \ref{case1}, because this is the only construction which allows $X_{\gamma_{n+1}}$ to contain a vertex from $X\setminus X_{\gamma_n}$ without implying in $w\in D_{\gamma_n}\subseteq D_{\Omega}$.  However, the normality of $T$ now requires that $\lceil w\rceil \cap D_{\gamma_{n+1}}\supsetneq \lceil w \rceil \cap D_{\gamma_n}$, since Case \ref{case1} applies only if there is a vertex in $D_{\gamma_{n+1}}\setminus D_{\gamma_n}$ which contains infinitely many neighbors in $\lfloor (X_{\gamma_{n+1}}\cap X)\setminus X_{\gamma_n}\rfloor$. Hence, we hit the contradiction that $\{\lceil w\rceil \cap D_{\gamma_n}\}_{n\in\mathbb{N}}$ should be an infinite strict $\subseteq-$increasing sequence of subsets of the finite set $\lceil w\rceil$.

		In its turn, we observe that there is indeed a vertex $w\in H\setminus D_{\Omega}$ which is strictly greater than $v$ in the tree order of $T$. If not, then $(\lfloor v\rfloor \setminus \{v\})\cap H\subseteq D_{\Omega}$, from where Case \ref{case5} can be used to describe a coloring $c_{\Omega+1}$ which contains $\lfloor v\rfloor$ in its domain, contradicting the choice of $\Omega$.

	Finally, we will now argue that there is a vertex $w\in H\setminus (D_{\Omega}\cup \lceil v\rceil)$ which has only finitely many neighbors in $H$. For instance, suppose that every vertex $w\in H\setminus (D_{\Omega}\cap \lceil v\rceil)$ also contains infinitely many neighbors in $H\setminus D_{\Omega}$. Then, $\{w \in H\setminus (D_{\Omega}\cup \lceil v\rceil): |N(w)\cap H\setminus (D_{\Omega}\cup \lceil v\rceil)|=\aleph_0\}$ is a non empty $\aleph_0$-regular subset of $H$ by its own definition and since $\lceil v\rceil$ is finite. However, giving rise to a contradiction, this subset should be contained in $D_0\subseteq D_{\Omega}$ due to the definition of the coloring $c_0$. Therefore, there is indeed a vertex $w\in H\setminus D_{\Omega}$ which is greater than $v$ in the tree order of $T$ and such that $N_{\hat{G}}(w)\cap H\setminus D_{\Omega}$ is finite. Finishing the proof, we note that $|N_{\hat{G}}(w)\cap H\cap D_{\Omega}|<\aleph_0$ as well: otherwise, Case \ref{case2} could be applied to describe a coloring $c_{\Omega+1}$ that contains $w$ in its domain, contradicting the choice of $\Omega$ once more.

\end{proof}

If we assume for a contradiction that $D_{\Omega}\neq V(\hat{G})$, then the root $r$ of $T$ is a vertex of $H$ that verifies $\lfloor r\rfloor \setminus D_{\Omega}\neq \emptyset$. Then, the above lemma claims that there is $w_0 \in H\setminus (D_{\Omega}\cup \lceil v\rceil)$ which contains only finitely many successors in $T$ and only finitely many neighbors in $H$. In particular, there is a successor $u_0\in S(w_0)$ such that $N(w_0)\cap \lfloor u_0\rfloor$ is infinite, since $w_0$ has infinite degree and $T$ is a normal tree. We note that $u_0\notin S$: otherwise, Case \ref{case1} provides a coloring $c_{\Omega+1}$ that contains $w_0$ in its domain, contradicting the choice of $\Omega$. Then, $u_0$ belongs to $H$ and $\lfloor u_0\rfloor \setminus D_{\Omega}\neq \emptyset$, where this last statement follows from the fact that Case \ref{case4} cannot be applied to define a coloring $c_{\Omega+1}$ that contains $w_0$ in its domain.

Hence, by induction over $n\in\mathbb{N}$, suppose that have so far defined a vertex $w_n \in H$ and a successor $u_n \in S(w_n)\cap H$ so that $N(w_n)\cap \lfloor u_n\rfloor$ is infinite, $w_n$ has finitely many neighbors in $H$ and $\lfloor u_n\rfloor \setminus D_{\Omega}\neq \emptyset$. Then, Lemma \ref{l36} claims that there is a vertex $w_{n+1} > u_n$ which belongs to $H\setminus D_{\Omega}$, it has finitely many neighbors in $H$ and such that $S(w_n)$ is finite. However, $w_{n+1}$ has infinite degree in $\hat{G}$, meaning that there is a successor $u_{n+1}\in S(w_{n+1})$ for which $|N(w_{n+1})\cap \lfloor u_{n+1}\rfloor|=\aleph_0$. As before, we note that $u_{n+1}\notin S$ since, otherwise, Case \ref{case1} could be applied to describe a coloring $c_{\Omega+1}$ that contains $w_{n+1}$ in its domain. The same conclusion would be reached by Case \ref{case4} if $\lfloor u_{n+1}\rfloor \subseteq D_{\Omega}$. In other words, $u_{n+1}\in H$ and $\lfloor u_{n+1}\rfloor \setminus D_{\Omega}\neq \emptyset$.

Once $\{w_n\}_{n\in\mathbb{N}}$ and $\{u_n\}_{n\in\mathbb{N}}$ are defined, we note there is a vertex $x_n\in N(w_n)\cap \lfloor u_n\rfloor \cap \lfloor S\rfloor$ for every $n\in\mathbb{N}$, since $N(w_n)\cap \lfloor u_n\rfloor$ is infinite but $w_n$ contains only finitely many neighbors in $H$. Unless by passing $\{w_n\}_{n\in\mathbb{N}}$ and $\{u_n\}_{n\in\mathbb{N}}$ to suitable subsequences, we may assume that $x_n$ and $w_{n+1}$ are incomparable in the tree order of $T$ for every $n\in\mathbb{N}$. After all, $\lceil x_n\rceil$ is finite and, thus, we can consider $w_{n+1}$ as the $\leq-$minimal vertex of $\{w_i\}_{i\in\mathbb{N}}$ in $G\setminus \lceil x_n\rceil$. Under this assumption, we have that the paths $x_nTw_{n+1}$ and $x_mTw_{m+1}$ are disjoint for any two natural numbers $n < m$, because the latter lies on $\lfloor u_{m}\rfloor$ and the former is contained in $\lfloor u_n\rfloor \setminus \lfloor u_{n+1}\rfloor\subseteq \lfloor u_n\rfloor \setminus \lfloor u_m\rfloor$. However, both paths contains vertices of $S$, since $x_n,x_m\in \lfloor S\rfloor$. Once $S$ comprises only vertice of finite degree by its definition, the ray described by the concatenation $w_0x_0Tw_1x_1Tw_2x_2Tw_3x_3\dots$ is thus alternating, which contradicts the main hypothesis over $G$ as fixed in the beginning of this section.

Therefore, $c_{\Omega}$ is globally defined and, as a stable coloring, it is strongly maximal in $\lfloor X_{\Omega}\rfloor$. Observing that this set contains all the vertices of finite degree in $G$, it follows that $c_{\Omega}$ is actually strongly maximal over its domain. Hence, the observation below finishes the proof of Theorem \ref{EnunciadoTecnico}:

\begin{prop}
	$c_{\Omega}$ is unfriendly in the vertices of infinite degree of $G$.
\end{prop}   
\begin{proof}
	Let $w\in V(G)$ be a vertex of infinite degree. If $w\in D_0$, then $c_{\Omega}$ is unfriendly in this vertex because do does $c_0$ and since $c_{\Omega}|_{D_0}=c_0$ by construction. If $w\notin D_0$, fix $\alpha:=\min\{\eta < \Omega : w \in D_{\eta}\}$, observing that this is a successor ordinal due to the definition of $D_{\alpha}$ as $\bigcup_{\beta < \alpha}D_{\beta}$ if $\alpha$ were a limit one. Therefore, $c_{\alpha+1}$ is constructed according to one of the Cases \ref{case3}-\ref{case4}.

	If $c_{\alpha}$ is described via Case \ref{case3}, there is $X\subseteq S$ such that $w \in \lfloor X\rfloor N_{\hat{G}}(\lfloor X\rfloor)\subseteq D_{\alpha}$. In addition, $|\{u \in N_{\hat{G}}(w)\cap \lfloor X\rfloor : c_{\alpha}(u)\neq c_{\alpha}(w)\}| = \aleph_0$. Observing that $c_{\Omega}$ agrees with $c_{\alpha}$ on $N_{\hat{G}}(\lfloor X\rfloor)\cup \lfloor X\rfloor$, once $N_{\hat{G}}(\lfloor X\rfloor)$ is contained in this set and comprises only vertices of infinite degree, it follows that $c_{\Omega}$ is unfriendly in $w$. The same conclusion holds if $c_{\alpha}$ is defined according to Case \ref{case2}: after all, the maps $\{c_{\beta}\}_{\beta \leq \alpha}$ assign the same colors to the vertices of $H$ in which they are defined.

	Finally, if $c_{\alpha}$ is described by one of other three cases (\ref{case1}, \ref{case5} or \ref{case4}), there is a vertex $v\in X_{\alpha}$ satisfying $|\{u\in N_{\hat{G}}(w)\cap \lfloor v\rfloor : c_{\alpha}(u)\neq c_{\alpha}(w)\}|=\aleph_0$. Fixing the ordinal $\beta = \min \{\alpha \leq \eta < \Omega : \lceil v\rceil \subseteq D_{\eta}\}$, it follows that $c_{\beta}|_{[v]}=c_{\Omega}|_{[v]}$, because the vertices of finite degree in $[v]$ and their neighborhoods are contained in $D_{\beta}$. On the other hand, $(c_{\beta}\triangle c_{\alpha})\cap [v]\subseteq \lfloor v\rfloor$ is a finite set of vertices of finite degree, since $\lceil v \rceil$ is finite and it contains $N_{\hat{G}}(\lfloor v\rfloor)$ by the normality of $T$. Therefore, $c_{\Omega}$ is unfriendly in $w$ due to the fact that $\{u\in N_{\hat{G}}(w)\cap \lfloor v\rfloor : c_{\Omega}(u)\neq c_{\Omega}(w) = c_{\beta}(w) = c_{\alpha}(w)\}$ is thus still infinite.
\end{proof} 

\section{Acknowledgments}
\paragraph{}

We acknowledge Bowler and Niel for sharing their work in \cite{niel}, where Lemma \ref{niel} and its proof can be found. We also thank Pitz for notifying us of this study regarding the Unfriendly Partition Problem.

\bibliography{RemarksCountableCase}

\begin{thebibliography}{1}

\bibitem{aharoni}
R.~Aharoni, E.~C. Milner, and K.~Prikry.
\newblock {Unfriendly partitions of a graph}.
\newblock {\em Journal of Combinatorial Theory, Series B}, 50(1):1--10, 1990.

\bibitem{berger}
E.~Berger.
\newblock Unfriendly partitions for graphs not containing a subdivision of an
  infinite clique.
\newblock {\em Combinatorica}, 37, 02 2017.

\bibitem{bruhn}
H.~Bruhn, R.~Diestel, A.~Georgakopoulos, and P.~Spr{\"{u}}ssel.
\newblock Every rayless graph has an unfriendly partition.
\newblock {\em Combinatorica}, 30(5):521--532, 2010.

\bibitem{cowanemerson}
R.~Cowan and W.~Emerson.
\newblock Proportional colorings of graphs.
\newblock {\em unpublished}, 1985.

\bibitem{diestellivro}
R.~Diestel.
\newblock {\em Graph theory}, volume 173 of {\em Graduate Texts in
  Mathematics}.
\newblock Springer, Berlin, fifth edition, 2018.

\bibitem{jung}
H.~A. Jung.
\newblock Wurzelbäume und unendliche {W}ege in {G}raphen.
\newblock {\em Mathematische Nachrichten}, 41(1-3):1--22, 1969.

\bibitem{niel}
C.~J. Niel.
\newblock {\em Unfriendly partition conjecture for countable graphs with no \(
  T{K^{\aleph_0}} \) minor}.
\newblock Bachelor's dissertation, Hamburg, Germany, 2023.

\bibitem{shelah}
S.~Shelah and E.~C. Milner.
\newblock Graphs with no unfriendly partitions.
\newblock In {\em A tribute to {P}aul {E}rd\H{o}s}, pages 373--384. Cambridge
  Univ. Press, Cambridge, 1990.

\end{thebibliography}
\bibliographystyle{plain}

\end{document}